%% file: eight_d_arxiv.tex
\documentclass[american,a4paper]{article}
\usepackage[T1]{fontenc}
\usepackage{float}
\usepackage{url}
\usepackage{amsmath}
\usepackage{amsfonts}
\usepackage{setspace}
\input macro
\makeatletter

\providecommand{\tabularnewline}{\\}

\newenvironment{svmultproof}{\begin{proof}}{\end{proof}}

\makeatother
\DeclareSymbolFont{rrfam}{U}{cmrr}{m}{n}
\SetSymbolFont{rrfam}{bold}{U}{cmrr}{b}{n}
\DeclareSymbolFontAlphabet\mathrrfam{rrfam}

\def\F{{\mathbb F}}

\def\T{{\mathbb T}}

\newcommand{\field}[1]{\mathbb{#1}}
\newcommand{\C}{\field{C}}
\newcommand{\R}{\field{R}}
\newcommand{\N}{\field{N}}
\newcommand{\D}{\field{D}}
\renewcommand{\S}{\field{S}}
\usepackage{babel}

\newtheorem{lemma}{Lemma} 
\newtheorem{proposition}{Proposition} 
\newtheorem{acknowledgements}{Acknowledgement} 
\newtheorem{rem}{Remark} 
\begin{document}

\title{Equivariant Bifurcation and Absolute Irreducibility in $\R^8$\\
{\normalsize A step towards the Ize Conjecture} %
\thanks{This paper is dedicated to the memory of Klaus Kirchg\"assner, whose
insights changed our understanding of nonlinear analysis and applied
mathematics. We also acknowledge the importance of the work of Jorge
Ize, whose contributions are fundamental to this research and who
passed away shortly before this work was finished. The research was
supported by DFG under LA525/11-1.%
}}

\author{Reiner Lauterbach\\Department Mathematics\\
University of Hamburg\\ Bundesstra\ss{}e 55\\ 20146 Hamburg\\
lauterbach@math.uni-hamburg.de}

\maketitle
\begin{abstract}
M. Field \cite{Fd6} refers to an unpublished work by J. Ize for a
result that loss of stability through an absolutely irreducible representation
of a compact Lie group leads to bifurcation of steady states. The
main ingredient of the proof is the hypotheses, that for an absolutely
irreducible representation of a compact Lie group there exists a closed
subgroup whose fixed point space is odd dimensional. Then, using Brouwer
degree, one gets the result. We refer to the hypotheses that for an
absolutely irreducible representation of a compact Lie group there
exists at least one subgroup with an odd dimensional fixed point space
as the algebraic Ize conjecture (AIC). Lauterbach and Matthews \cite{LM}
have shown that the (AIC) is in general not true. In fact they have
constructed three infinite families of finite subgroups of $\SO{4}$
which act absolutely irreducibly on $\R^4$ and for each of them any
isotropy subgroup has an even dimensional fixed point space. Moreover
in \cite{LM} it is shown that in spite of this failure of the (AIC)
the original conjecture is true at least for groups in two of these
three families. In this paper we show a similar bifurcation result
for the third family defined in \cite{LM}. We go on and construct
a family of groups acting absolutely irreducibly on $\R^8$ which
have only even dimensional fixed point spaces. Then we discuss the
steady state bifurcations in this case. Key ingredients are an abstract 
group theoretic construction and a kind of inductive step reducing the 
issue of bifurcations to a problem in $\R^4$.
We end this paper with a discussion
on how to extend the results in \cite{LM} to larger sets of groups
which act on $\R^{4}$ and $\R^8$. In this context we point out, that the inductive step, which is important our arguments, does not work in general and this gives rise to interesting new questions. 

\end{abstract}

\section{Introduction\label{sec:Introduction}}
In the early days of equivariant bifurcation theory various versions
of the Equivariant Branching Lemma (see for example \cite{V1,Sat1,Cic1,Sat2,IG,GSS,ChL}
for early results and some newer fomulations) have been established
and successfully applied in many different applications. 
 Basically it says that we look for the action of the group in question on the kernel and for each subgroup having a one dimensional fixed point space we find symmetry breaking solutions. So it is a natural question to ask whether there are always such subgroups. A more general question is to ask whether there are subgroups having odd dimensional fixed point spaces. This is the context of the Ize conjecture, which we will discuss in more detail in section \ref{sec:IzeConject}. Here we continue previous work by 
Lauterbach and Matthews \cite{LM}. They have used the biquaternionic
representation of elements in $\SO{4}$ to construct three families
of groups $\mathcal{F}_j=\{G_j(m)\}_{m\in (2\N+1)}$,\mbox{ for } $j=1,2,3$
of orders $16m+32$, i.e. of orders $48,80,112, \dots$, such that
the natural action on $\R^4$ is absolutely irreducible and each of
these groups has up to conjugacy precisely $j$ nontrivial isotropy
subgroups and each of them has a two dimensional fixed point subspace.
It is also shown there that for $j=1,2$ and for bifurcation problems which are equivariant with respect to these representations these fixed point subspaces contain
generically nontrivial branches of solutions. Here, generically means
the bifurcation occurs for an open and dense subset of the cubic order
bifurcation equations and these are stable under higher order perturbation.
In the language of Field \cite{Fd6,Fd-07} the bifurcations are $3$-determined
and the group actions are symmetry breaking. 

There are various ways of specifying the groups, one could either
give the biquaternionic generators, one could use the classification of Conway
and Smith \cite{CS}, or one could use the methodology of GAP \cite{GAP} to
denote the groups.  
In \cite{LM} we have used all three descriptions
and we have seen how to go from one to the other. In the present paper
the main emphasis is on group actions on $\R^8$ and therefore we
will use matrices to define the groups in question. In addition we
shall provide the GAP-names of our groups, at least for those of low order. This enables the reader to have a quick check of some of our (lengthy) computations at least for the groups of small order within our family. We define a family of groups
which relates in a rather straight forward way to the family $\mathcal{F}_3$.
In \cite{LM} we have not investigated the bifurcation behavior for
this family. In this respect We will complete the results of \cite{LM}   in  section \ref{sec:Bifurcations4} which follows a short discussion of the Ize conjecture in the next section. Then
we will give the construction of the new eight dimensional family
and we discuss the bifurcations for this new family afterwards.

Finally we remark that a minor change in the definition of the three series
in \cite{LM} gives much larger families and these new families contain
most of the (computationally known) four dimensional counter examples to the (AIC)
(compare section \ref{sec:Remarks}). 
There are many infinite series
in $\R^{8}$. It is by far not clear how to order them in a reasonable
form. Among other phenomena we find series of groups having only one
nontrivial isotropy type and the corresponding isotropy subgroup
has a four dimensional fixed point space. The dynamics in such a case
is not yet understood. We will point out some of this in the last section of this paper. 
\section{The Ize conjecture}\label{sec:IzeConject}
In classical bifurcation theory it was a question whether a loss a stability of a given branch of steady states of a nonlinear equation through an eigenvalue zero will lead to to the creation of new branches. It is known that one can easily construct examples where this is not the case. However Crandall and Rabinowitz \cite{CR1} showed that if the eigenvalue is simple, then there is a steady state bifurcation in the sense that there is a new branch of solutions bifurcating from the given one. \\ 
In the context of equivariant bifurcation theory there is a similar question, i.e. the question what is the correct generalization of a simple eigenvalue zero.    It seems natural to require minimal degeneracy and this is to require an absolutely irreducible representation of the group on the kernel at the corresponding point. So it is a natural conjecture, that absolutely irreducible group actions on the kernel lead to bifurcation of branches of relative equilibria which are in a natural way  the objects corresponding to steady states in the non-equivariant context. This version of the conjecture has been around for some time. It was mentioned to the author by Marty Golubitsky in 2005. In his 1996 notes M. Field \cite{Fd6} attributes a similar conjecture to Jorge Ize, we quote a footnote from \cite{Fd6}, p.63 : "`It follows from recent work of Ize [50], that every absolutely irreducible representation has an odd dimensional fixed point space and so, using results of section 4, has a generically symmetry breaking isotropy type."'. We refer to this statement on symmetry breaking as Ize conjecture (IC).  This quote also provides a possible proof: if $G$ is a compact Lie group acting absolutely irreducibly on a finite dimensional real space, then there exists a subgroup $H$ whose fixed point space is odd dimensional. This last statement is sufficient to prove the Ize conjecture, however it is by no means necessary. It is a purely algebraic statement, so we refer to this statement as the algebraic Ize conjecture (AIC). There is is some evidence that the (AIC) is true in dimensions of the form $2\mod 4$. First it is obviously true in dimension $2$, Ruan \cite{Ruan10} proves the validity in dimension 6 under mild additional assumptions. Our GAP computations have not found counterexamples to the (AIC) in dimensions $6,\ 10,\ 14$ and $18$.   
\section{Bifurcations for the  family $\mathcal{F}_3$ in $\R^4$\label{sec:Bifurcations4}}
\subsection{The Groups in the third family $\mathcal{F}_3$  \label{sub:Groups-third-series}}
We employ the methods from \cite{LM} to discuss the bifurcation behavior.
We use the standard basis $i^{2}=j^{2}=k^{2}=-1$ for the quaternions
and write $e_{m}$ for one of the $m$-th primitive roots of $-1$ in $\C$. With
this notation we recall the definition of the group 
\begin{equation}\label{equ:G3_generation}
G_{3}(m)=\langle[e_{m},1],[1,i],[j,1],[1,j]\rangle,\ m\in \N, m \mbox { odd }
\end{equation}
using pairs of quaternions of length $1$, where we identify $[a,b]$ and $[-a,-b]$.
Each such pair represents an element in $\SO{4}$, see \cite{CS}.  
Observe the action of the pair $[a,b]$ on $v=v_1+v_2i+v_3j+v_4k$ is given by
\[[a,b]v\mapsto \bar a v b.\]
We recall Theorem
2.1 from \cite{LM}. In the context of our family of groups, it says
that for each $m\in 2\N+1$ the group $G_{3}(m)$ has the order $16m$,
if $m'$ divides $m$ then $G_{3}(m')\subset G_{3}(m)$, the closure
of the union of all these groups is a compact, $1$-dimensional Lie
group $G_3$.  Theorem 2.2 of \cite{LM}
tells us that the action of $G_{3}(m)$, for an odd integer $3\le m\in\N$,
is absolutely irreducible, the same applies to $G_{3}$. For each
group  $G_{3}(m)$ and for $G_{3}$ there are (up to
conjugation) precisely three nontrivial isotropy subgroups, each of
them has a two dimensional fixed point space. These groups do not depend on $m$ (however the number of conjugates is unbounded as $m\to\infty$). 
In biquaternionic notation these groups are 
\[H_1=\langle[j,i]\rangle,\ H_2=\langle[j,j]\rangle,\ H_3=\langle[j,k]\rangle.\]
In each case the fixed point space is two dimensional. For the fixed point subspace of $H_j$, $j\in\{1,2,3\}$ we find
\begin{enumerate}
\item $H_1=\langle[j,i]\rangle$
\[[j,i](v_1+v_2i+v_3j+v_4k)=v_4+v_3i+v_2j+v_1k,\]
and therefore
\[\Fix{H_1}=\Meng{v\in\R^4}{v_1=v_4\mbox{ and }v_2=v_3},\]
\item $H_2=\langle[j,j]\rangle$
\[[j,j](v_1+v_2i+v_3j+v_4k)=v_1-v_2i+v_3j-v_4k,\]
therefore the fixed point space is given by 
\[\Fix{H_2}=\Meng{v\in\R^4}{v_2=v_4=0},\] 
\item $H_3=\langle[j,k]\rangle$
\[[j,k](v_1+v_2i+v_3j+v_4k)=-v_2-v_1i+v_4j+v_3k,\]
and therefore its fixed point space is given by
\[\Fix{H_3}=\Meng{v\in\R^4}{v_1=-v_2\mbox{ and }v_3=v_4 }\]
\end{enumerate}
\subsection{The equivariant structure\label{sub:equivariant-structure}}
The next step is to consider the structure of the cubic $G_3$ (or $G_3(m)$)-equivariant
maps $\R^{4}\to\R^{4}$. In Theorem 4.1 in \cite{LM} it is shown
that the space of cubic equivariant maps is three dimensional. For
the determinacy and bifurcation results we only need the non-radial
equivariant maps (see the work by Field \cite{Fd6,Fd-07}). A basis for the space of non-radial
cubic equivariant maps for the groups $G_{3}(m)$  and the group  $G_{3}$
consists of gradients of invariant quartic polynomials and these have
the form 
\[
I_{4,1}(z_{1},z_{2})=\frac12|z_{1}|^{2}|z_{2}|^{2},\qquad I_{4,2}(z_{1,}z_{2})=\frac12\left(z_{1}^{2}{\bar{z}}_{2}^{2}+\mbox{\ensuremath{\bar{z}}}_{1}^{2}z_{2}^{2}\right),
\]
where $z_{\mu}=x_{\mu}+iy_{\mu}$ for $\mu=1,2$ and $\R^{4}\ni v=z_{1}+z_{2}j$.
For $v\in\R^{4}$ we get the (real) gradients in Equ. (\ref{eq:real gradients}),
where $\rho_{1}=v_{1}^{2}+v_{2}^{2}$, $\rho_{2}=v_{3}^{2}+v_{4}^{2}$,
$\sigma_{1}=v_{1}^{2}-v_{2}^{2}$,$\;\sigma_{2}=v_{3}^{2}-v_{4}^{2}$
and $\tau_{1}=v_{1}v_{2},\;\tau_{2}=v_{3}v_{4}$ after rewriting $I_{4,\nu},\ $$\nu=1,2$
as
\[
I_{4,1}(v)=\frac12\rho_{1}\rho_{2},\:\ I_{4,2}(v)=\frac12(\sigma_{1}\sigma_{2}+4\tau_{1}\tau_{2}).
\]
Then taking the gradients we obtain (up to scalar multiples)
the equivariant maps
\begin{equation}
e_{3,1}(v)=\left(\begin{array}{c}
\rho_{2}v_{1}\\
\rho_{2}v_{2}\\
\rho_{1}v_{3}\\
\rho_{1}v_{4}
\end{array}\right),\quad e_{3,2}(v)=\left(\begin{array}{c}
\sigma_{2}v_{1}+2v_{2}\tau_{2}\\
-\sigma_{2}v_{2}+2v_{1}\tau_{2}\\
\sigma_{1}v_{3}+2\tau_{1}v_{4}\\
-\sigma_{1}v_{4}+2\tau_{1}v_{3}
\end{array}\right).\label{eq:real gradients}
\end{equation}
\subsection{Phase vectorfield and bifurcation\label{sub:Phase-vectorfield}}
In order to compute the phase vector field (see \cite{Fd6,Fd-07})
we restrict these  equivariants to the unit sphere in $\R^{4}$, compute the scalar
product with the unit radial vector and project $e_{3,\mu}$ for $\mu=1,2$
onto the tangent bundle to the sphere by subtracting the radial part:
\[
t_{3,\mu}(v)=e_{3,\mu}(v)-\langle e_{3,\mu}(v),\ v\rangle\ v,\:\ \ \mu=1,2.
\]
We obtain the two tangent fields
\[
t_{3,1}(v)=\left(\begin{array}{c}
\rho_2(\rho_2-\rho_1)v_{1}\\
\rho_2(\rho_2-\rho_1)v_{2}\\
\rho_1(\rho_1-\rho_2)v_{3}\\
\rho_1(\rho_1-\rho_2)v_{4}
\end{array}\right)\] 
\[t_{3,2}(v)=\left(\begin{array}{c}
\sigma_2(1-2\sigma_1)v_{1}+2\tau_{2}(v_2-4\tau_{1}v_1) \\
\sigma_2(-1-2\sigma_1) v_{2}+2\tau_{2}(v_1-4\tau_{1}v_2)\\
\sigma_1(1-2\sigma_2)v_{3}+2\tau_{1}(v_4-4\tau_{2} v_{3})\\
\sigma_1(-1-2\sigma_2)v_{4}+2\tau_{1}(v_3-4\tau_{2} v_{4})
\end{array}\right).
\]
We solve the equation $a\, t_{3,1}(v)+b\, t_{3,2}(v)=0$ for $a,b\in\R$. We want
to show that generically in $a,b\in\R$ the solutions are isolated
and (normally) hyperbolic. Without loss of generality we assume $b=1$.
In order to solve the remaining equation we restrict to one of the fixed point
subspaces. It suffices to show hyperbolicity of the solutions
on the sphere in the fixed point subspace, since we use a slightly more  generalized
version of the results in \cite{Fd6} to apply in fixed point subspaces. The fixed point spaces are two-dimensional, the intersection with the unit sphere gives a $S^1$. 
So we have  a map of the form 
\[S^1\to TS^1:\phi\mapsto f(x(\phi),y(\phi))\left( \begin{array}{c}
    y(\phi)\\-x(\phi)
  \end{array}\right),\ \phi\in S^1, \left(
  \begin{array}{c}
    x\\y
  \end{array}\right):S^1\to\R^2
\]
and zeros are regular if $f(x(\phi),y(\phi))=0$ implies, that $Df(x(\phi),y(\phi))(x(\phi),y(\phi))^T\not=0$. In all cases we reduce the equivariant map to this form. \\[3mm] 
The fixed point space for $H_1$ is given by $\Fix{H_1}=\Meng{v\in\R^4}{v_1=v_4\mbox{ and }v_2=v_3}$.
This means for $v\in\Fix{H_1}$ we have $\rho_1=\rho_2$, $\sigma_1=-\sigma_2$ and $\tau_1=\tau_2$. Therefore $t_{3,1}(v)=0$ on this space.
For the mapping $t_{3,2}$ restricted to $\Fix{H_1}$ it is easily seen, that the first and the fourth equation and the second and third equation are the same, so we solve the first  and third equations and identify the variables:
\begin{eqnarray*}
 0&=& -\sigma_1(1-2\sigma_1)v_1+2\tau_1(v_2-4\tau_1v_1)\\
 0&=& \sigma_1( 1-2\sigma_2)v_1+2\tau_1(v_1-4\tau_2v_2).
\end{eqnarray*}
After a bit of  computation this reduces to
\[8\left(
  \begin{array}{c}
    v_1v_2^4-v_1^3v_2^2\\
   v_1^4v_2-v_1^2v_2^3
  \end{array}\right)=8(v_1v_2^3-v_1^3v_2)\left(
    \begin{array}{c}
      v_2\\-v_1
    \end{array}\right)\]
The zeros on the circle are given by $v_1=0, v_2=\pm\frac12$, $v_1=\pm\frac12,\ v_2=0$ and $v_1=\pm v_2$. These solutions are obviously regular. 
\\[3mm] 
We have seen $\Fix{H_2}=\Meng{v\in\R^4}{v_2=v_4=0}$.
This implies $\rho_{1}=\sigma_{1}=v_{1}^{2}$, $\rho_{2}=\sigma_{2}=v_{3}^{2}$
and $\tau_{1}=\tau_{2}=0$. So the second and last equation are identically
satisfied and the first and third equation take the form (using $1-\rho_{j}=\rho_{3-j},\ j=1,2$)
\[
a\, v_{3}^{2}(v_{3}^{2}-v_{1}^{2})v_{1}+v_{3}^{2}\left(v_{3}^{2}-v_{1}^{2}\right)v_{1}=0
\]

\[
a\, v_{1}^{2}(v_{1}^{2}-v_{3}^{2})v_{3}+v_{1}^{2}(v_{3}^{2}-v_{3}^{2})v_{3}=0.
\]
Observe that from the assumption that $\|v\|=1$ we cannot have that
$v_{1}=v_{3}=0$. A simplification yields 
\[
(a+1)v_{3}^{2}(v_{3}^{2}-v_{1}^{2})v_{1}=0
\]
\[
(a+1)v_{1}^{2}(v_{1}^{2}-v_{3}^{2})v_{3}=0,
\]
or 
\[(a+1)v_1v_3(v_3^2-v_1^2)\left(
  \begin{array}{c}
    v_3\\-v_1
  \end{array}\right)=0\]

We obtain zeros at $v_{1}=0$, $v_{3}=0,$ and for $v_{1}^{2}=v_{3}^{2}$.
Each zero is regular on $v_{1}^{2}+v_{3}^{2}=1$. \\[3mm]
In $\Fix{H_3}=\Meng{v\in\R^4}{v_1=-v_2\mbox{ and }v_3=v_4}$ we have
\[\rho_1=2v_1^2,\ \rho_2=2v_3^2,\ \sigma_{1,2}=0,\tau_1=-v_1^2,\ \tau_2=v_3^2.\]
So we rewrite the fields $t_{3,1}, t_{3,2}$ as
\[t_{3,1}(v)=\left(
  \begin{array}{c}
    v_3^2(v_3^2-v_1^2)v_1\\
    -v_3^2(v_3^2-v_1^2)v_1\\
    v_1^2(v_1^2-v_3^2)v_3\\
 v_1^2(v_1^2-v_3^2)v_3\\
  \end{array}\right)\]
\[t_{3,2}(v)=\left(
  \begin{array}{c}
   -4v_3^2(v_1^2+v_3^2)v_1+8v_1^2v_3^2v_1\\
    4v_3^2(v_1^2+v_3^2)v_1-8v_1^2v_3^2v_1\\
   -4v_1^2(v_1^2+v_3^2)v_3+8v_1^2v_3^2v_3\\
    4v_1^2(v_1^2+v_3^2)v_3-8v_1^2v_3^2v_3\\
\end{array}
\right)\]
We have to look at the linear combination and there at the first and third equation to get 
\[(a+4)v_1v_3(v_1^2-v_3^2)\left(
  \begin{array}{c}
    v_3\\
    -v_1
  \end{array}\right).\]
Again we find four isolated solutions.\\[3mm]
By Theorem 5.2.1
in \cite{Fd6} we have the following result.
\begin{theorem}{Theorem A}
The four dimensional irreducible representation of the groups in the
family $\mathcal{F}_3$ including $G$, is symmetry breaking in the sense of {\rm \cite{Fd6}}.
All three isotropy types $[H_1],[H_2],[H_3]$ are symmetry breaking. The corresponding solution
branches contain steady states.
\end{theorem}
\begin{rem}
  The cubic equivariant maps are equivariant with respect to the Lie group $G_3$. Therefore the zero set of these equivariant cubics contain (topological) circles. For the groups $G_3(m)$ there are higher order terms which break this $G_3$-symmetry and should lead to isolated zeros. If this were the case, methods developed by Field \cite{Fd6,Fd-07} could lead to more precise bifurcation results especially to establish the existence of branches in trivial isotropy. We have made no attempt to determine the higher order equivariants which break the continuous symmetry. 
\end{rem}
\section{Construction of groups \label{sec:Construction-of-groups}}
In this section we describe the construction of a family
of groups acting absolutely irreducibly in dimension eight which show that the (AIC) fails in this dimension. 
The simplest way to define these groups
is to construct some $2\times2$-matrices, use them to built $4\times4$-matrices
and finally we use these to form several $8\times8$-matrices. These
will be the generators of the groups in question. 
For $k\in\N$ let $\zeta_{k}\in\C$ be one of the primitive $k$-th
roots of unity. We set
\[
d_{1}(k)=\frac{1}{2}\left(\zeta_{k}+\zeta_{k}^{k-1}\right),\: d_{2}(k)=\frac{1}{2i}\left(\zeta_{k}-\zeta_{k}^{k-1}\right),\: f=\frac{1}{2}\left(\zeta_{8}-\zeta_{8}^{3}\right).
\]
We define the following $2\times2$ matrices (for $k=4+8\ell$, $\ell=1,2,3,\dots$)
\[
F_{1}=\left(\begin{array}{cc}
-f & -f\\
-f & f
\end{array}\right),\, F_{2}=\left(\begin{array}{cc}
-f & f\\
f & f
\end{array}\right),\,\1_{2}=\left(\begin{array}{cc}
1 & 0\\
0 & 1
\end{array}\right),\] \[D_k=\left(\begin{array}{cc}
-d_{2}(k) & -d_{1}(k)\\
d_{1}(k) & -d_{2}(k)
\end{array}\right),\, 0_{2}=\left(\begin{array}{cc}
0 & 0\\
0 & 0
\end{array}\right).
\]
\[
T_{1}=\left(\begin{array}{cc}
-1 & 0\\
0 & 1
\end{array}\right),\ T_{2}=\left(\begin{array}{cc}
1 & 0\\
0 & -1
\end{array}\right),\ S_{1}=\left(\begin{array}{cc}
0 & -1\\
-1 & 0
\end{array}\right),\ S_{2}=\left(\begin{array}{cc}
0 & 1\\
1 & 0
\end{array}\right).
\]
For later use, we note that $T_\mu$, $S_\mu$, $F_\mu$ for $\mu=1,2$ are reflections and $D_k$ is a rotation in $\R^2$. \\
Now we use these to define some $4\times4$-matrices:
\[
\F=\left(\begin{array}{cc}
0_{2} & F_{1}\\
F_{2} & 0_{2}
\end{array}\right),\;\S=\left(\begin{array}{cc}
S_{1} & 0_{2}\\
0_{2} & S_{2}
\end{array}\right),\;\T=\left(\begin{array}{cc}
0_{2} & T_{1}\\
T_{2} & 0_{2}
\end{array}\right),\;\D_k=\left(\begin{array}{cc}
D_k & 0_{2}\\
0_{2} & D_k
\end{array}\right),\;0_4=\left(\begin{array}{cc}
0_2 & 0_{2}\\
0_{2} & 0_2
\end{array}\right).
\]
From this we get the following $8\times8$-matrices
\[
R_{1}=\left(\begin{array}{cc}
0_{4} & \F\\
\F & 0_{4}
\end{array}\right),\; R_{2}=\left(\begin{array}{cc}
\S & 0_{4}\\
0_{4} & \T
\end{array}\right)\:\mbox{{and}}\; R_{3}(k)=\left(\begin{array}{cc}
\D_k & 0_{4}\\
0_{4} & \D_k
\end{array}\right).
\]
~\\[3mm]
{\bf Definition 1}\\
For $\ell\in\N$ we set $k=4+8\ell$ and define
\[
G(\ell)=\langle R_{1},R_{2},R_{3}(k)\rangle.
\]
~\\[3mm]
\begin{theorem}{Theorem B}
For $\ell\in\N$ we have
\begin{enumerate}
\item $G(\ell)$ is a group of order $16k=64+128\ell$, i.e. the orders
are $192,320,448,\dots$.
\item $G(\ell)$ acts absolutely irreducibly on $\R^{8}$.
\item If $k_j=4+8l_j$, $j=1,2$ and if $k_{1}|k_{2}$ then $G(\ell_{1})\subset G(\ell_{2})$ and hence
\[
G=\overline{\bigcup_{\ell\in\N}G(\ell)}
\]
is a compact Lie group acting absolutely irreducibly on $\R^{8}$. 
\item If $H<G(\ell)$ or $H<G$ is an isotropy subgroup, then $\dim\Fix{H}$
is even. \end{enumerate}
\end{theorem}
\begin{rem}
 $G$ is a compact, $1$-dimensional Lie group which has $8$ components. This can be seen from the structure of $G(\ell)$ when $\ell\to\infty$, then we have always eight cosets of the cyclic group generated by $R_3(k)$. The infinitessimal generator of the component of the identity in $G$ is given by a $2\times 2$-block diagonal matrix with block sof the form $\left(
   \begin{array}{cc}
     0&-1\\
     1&0
   \end{array}\right)$. 
\end{rem}
In order to prove this result we need a couple of lemmas which are needed to verify the hypotheses of the Abstract Generation Theorem (Theorem C). This theorem is the major technical part of the proof. We will state this theorem, which might be of  some  interest  by itself, in the next section and provide the proof of it, which is independent of all other considerations in this paper, in section \ref{sec:pf_thm_C}.
\begin{lemma}\label{lem:rel1}
$R_1^8=\1_8$
\end{lemma}
\begin{svmultproof}
Observe that
\[
R_{1}^{2}=\left(\begin{array}{cc}
\F^{2} & 0_{4}\\
0_{4} & \F^{2}
\end{array}\right),
\]
 where 
\[
\F^{2}=\left(\begin{array}{cc}
0_{2} & F_{1}F_{2}\\
F_{2}F_{1} & 0_{2}
\end{array}\right)
\]
 and
\[
F_{1}F_{2}=\left(\begin{array}{cc}
0_{2} & -2f^{2}\\
2f^{2} & 0_{2}
\end{array}\right)=-F_{2}F_{1}.
\]
Since $2f^{2}=1$ we get immediately that $R_{1}^{4}=-\1_{8}\in C(\SO{8})$,
the center of $\SO{8}$. This implies $R_1^8=\1_8$.
\end{svmultproof}
The remaining computations are somewhat simpler if we rewrite the set of generators.
We set (suppressing the dependence on $k$)
\[
A=R_{1}\cdot R_{2}\cdot R_{3}(k).
\]
Considering the structure of the matrices
it follows that $A$ has
the form (where $\D=\D_k$)
\begin{equation}
A=\left(\begin{array}{cc}
0_{4} & \F\T\D\\
\F\S\D & 0_{4}
\end{array}\right).\label{eq:represent1_A}
\end{equation}
We get 
\[
A^{2}=\left(\begin{array}{cc}
\F\T\D\F\S\D & 0_{4}\\
0_{4} & \F\S\D\F\T\D
\end{array}\right).
\]
From this we get an explicit form for $A^{4}$ and for $A^{8}.$ 
The following lemmas concern orders of
some elements, some relations and some generating properties.  
\begin{lemma}\label{lem:rel2}
Fix $\ell\in\N$, let $k=4+8\ell=4\tau$, then we have
\begin{enumerate}
\item  $A^{8}=R_{3}^{8}$,
\item $R_3^{2\tau}=\1_8$,
\item if $\tau\equiv 3\mod 4$ then $R_3^\tau=\1_8$, 
\item that the order of $A$ is $2k$.
\end{enumerate}
\end{lemma}
\begin{svmultproof}
In the first step we show the relation between $A$ and $R_{3}.$
Here we refine the Equation (\ref{eq:represent1_A}) writing $2\times2$-blocks
instead of $4\times4$-blocks: a short calculation yields 
\begin{equation}\label{eq:represent4_A}
A=\left(\begin{array}{cccc}
0_{2} & 0_{2} & F_{1}T_{2}D & 0_{2}\\
0_{2} & 0_{2} & 0_{2} & F_{2}T_{1}D\\
0_{2} & F_{1}S_{2}D & 0_{2} & 0_{2}\\
F_{2}S_{1}D & 0_{2} & 0_{2} & 0_{2}
\end{array}\right)
\end{equation}
Let us note that the linear maps $F_{\nu}S_{\mu}$ or $F_{\nu}T_{\mu}$ for $\mu,\nu\in\{1,2\}$
are rotations, and since $\SO{2}$ is abelian these linear maps commute
with $D$. This remark will be useful for several points in our proof.
Looking at $A^2$ and using this observation we obtain
\begin{eqnarray*}
A^{2}&=&\left(\begin{array}{cccc}
0_{2} & F_{1}T_{2}D_kF_{1}S_{2}D_k & 0_{2} & 0_{2}\\
F_{2}T_{1}D_kF_{2}S_{1}D_k & 0_{2} & 0_{2} & 0_{2}\\
0_{2} & 0_{2} & 0_{2} & F_{1}S_{2}D_kF_{2}T_{1}D_k\\
0_{2} & 0_{2} & F_{2}S_{1}D_kF_{1}T_{2}D_k & 0_{2}
\end{array}\right)\\
&=&\left(\begin{array}{cccc}
0_{2} & F_{1}T_{2}F_{1}S_{2}D_k^2 & 0_{2} & 0_{2}\\
F_{2}T_{1}F_{2}S_{1}D_k^2 & 0_{2} & 0_{2} & 0_{2}\\
0_{2} & 0_{2} & 0_{2} & F_{1}S_{2}F_{2}T_{1}D_k^2\\
0_{2} & 0_{2} & F_{2}S_{1}F_{1}T_{2}D_k^2 & 0_{2}
\end{array}\right).
\end{eqnarray*}
The products in front of the $D_k^2$ give $\pm\1_2$ and from here we see 
\[
A^{8}=\left(\begin{array}{cccc}
D_k^{8} & 0_{2} & 0_{2} & 0_{2}\\
0_{2} & D_k^{8} & 0_{2} & 0_{2}\\
0_{2} & 0_{2} & D_k^{8} & 0_{2}\\
0_{2} & 0_{2} & 0_{2} & D_k^{8}
\end{array}\right).
\]
Therefore $A^{8}=R_{3}(k)^{8}$.\\[2mm]
In the second step we prove the statement on the order of $R_3$. For
$k=12,\ 20,\ 28,\dots$ we define $\tau\in\N$, $\tau$ odd by
\begin{equation}
k=4\tau.\label{eq:define_tau}
\end{equation}
The first step is to show that
\begin{equation}
R_{3}^{2\tau}=\1_{8}.\label{eq:R3^2tau}
\end{equation}
 The matrix $D_k$ is a rotation in $\R^{2}$ by the angle $\frac{\pi}{2}+\arg(\zeta_{k})$
by elementary geometry. Choosing $\zeta_k$ such that $\arg(\zeta_k)=\frac{2\pi}{k}$ we get
\[
\frac{\pi}{2}+\frac{\pi}{2\tau}=\frac{\tau+1}{2\tau}\mbox{\ensuremath{\pi}.}
\]
If $\tau=3\mod4$ this has the form $2\frac{s}{\tau}\mbox{\ensuremath{\pi}}$
for some $s\in\N$ and hence the $\tau$-th power of $R_3$ is equal to $\1_{8}.$ This is the third assertion of the lemma. 
If $\tau=1\mod4$ then
\[
D^{2\tau}=\1_{2.}
\]
This proves the claim Equation (\ref{eq:R3^2tau}). \\[2mm]
Now we look at
\[
A^{2k}=A^{8\tau}=R_{3}^{8\tau}=(R_{3}^{2\tau})^{4}=\1_{8}.
\]
This proves that the order of $A$ divides $2k$. Since the order
of $A$ is a multiple of $8$ and a multiple of the odd number $\tau$
its order is $2k$. If we choose a $\zeta_k$ to be one of the other primitive $k$-th roots of unity, the argument can be easily adjusted.
This completes the proof of the lemma.
\end{svmultproof}
\begin{lemma}\label{lem:rels}
We have the following relations:
\begin{enumerate}
\item $\left(AR_{1}\right)^{2}=\left(R_{1}A\right)^{2}=\1_{8}$
\item $\left(A^{3}R_{1}^{3}\right)^{2}=\left(R_{1}^{3}A^{3}\right)^{2}=\1_{8}$
\item $R_{1}^{2}A^{2}R_{1}^{2}=A^{2}$.
\end{enumerate}
\end{lemma}
\begin{proof}
We prove the first and the third statement in detail, the second one
follows in a similar form. Recall
\begin{equation}
R_{1}=\left(\begin{array}{cc}
0_{4} & \F\\
\F & 0_{4}
\end{array}\right),\ \mbox{ where }\F=\left(\begin{array}{cc}
0_{2} & F_{1}\\
F_{2} & 0_{2}
\end{array}\right)\mbox{.}\label{eq:repr_R1}
\end{equation}
 For $R_{1}A$ we obtain from Equ. (\ref{eq:represent4_A})
\[
R_{1}A=\left(\begin{array}{cccc}
F_{1}F_{2}S_{1}D & 0_{2} & 0_{2} & 0_{2}\\
0_{2} & F_{2}F_{1}S_{2}D & 0_{2} & 0_{2}\\
0_{2} & 0_{2} & 0_{2} & F_{1}F_{2}T_{1}D\\
0_{2} & 0_{2} & F_{2}F_{1}T_{2}D & 0_{2}
\end{array}\right).
\]
Squaring this expression gives the squares of the two nontrivial entries
in the upper left corner. These entries are elements in $\OO{2}\setminus\SO{2}$
and hence their squares are $-\1_{2}$. For the entries in the lower
right corner we get their products on the diagonal. The entries are
again in $\OO{2}\setminus\SO{2}$. From $F_1F_2=-F_2F_1$ it follows that both
entries are the same and hence the products are equal to $\1_{2}$. \\{3mm}
The proof of the second statement is similar to the one we have just
seen, so let us directly go to the third statement. We write $A^{2}$ in $2\times2$-block
form as 
\[
A^{2}=\left(\begin{array}{cccc}
0_{2} & M_{1} & 0_{2} & 0_{2}\\
M_{2} & 0_{2} & 0_{2} & 0_{2}\\
0_{2} & 0_{2} & 0_{2} & M_{3}\\
0_{2} & 0_{2} & M_{4} & 0_{2}
\end{array}\right),\mbox{ where }M_{\nu}\mbox{ are rotations in }\R^{2}\mbox{ for }\nu=1,\dots,\ 4.
\]
$R_1^2$ has the form 
\begin{equation}\label{equ:r1square}
R_{1}^{2}=\left(\begin{array}{cccc}
\sigma & 0_{2} & 0_{2} & 0_{2}\\
0_{2} & \sigma^{-1} & 0_{2} & 0_{2}\\
0_{2} & 0_{2} & \sigma & 0_{2}\\
0_{2} & 0_{2} & 0_{2} & \sigma^{-1}
\end{array}\right),\mbox{ where }\sigma=F_{1}F_{2}=\left(\begin{array}{cc}
0 & -1\\
1 & 0
\end{array}\right)\mbox{ is a rotation}.
\end{equation}
 By our previous remark any rotations commute and it is a simple computation
to check that $R_1^{2}A^{2}R_1^{2}$ has a similar form as $A^{2}$ where
$M_{\nu}$ is replaced by $\sigma M_{\nu}\sigma^{-1}$ (or by
$\sigma^{-1} M_{\nu}\sigma$). But this is equal to $M_{\nu}$ (by
the commutativity) and hence the lemma is established. 
\end{proof}
\begin{lemma}\label{lem:gens}
For all $\ell\in\N$ we have  $G(\ell)=\langle R_{1},A\rangle$, the group generated by $R_{1}$ and $A$.
\end{lemma}
\begin{svmultproof}
Fix $k\in\N$, and $\tau$ odd as in Equ. (\ref{eq:define_tau}).
The argument is  different for the two cases $\tau\equiv1\mod4$
and $\tau\equiv3\mod4$.\\[2mm] 
In the second case we have $R_{3}=R_{3}(k)$
is equal to $R_{3}=\left(R_{3}^{8}\right)^{s}$, where $s$ is chosen
such that $8s=1\mod\tau$. It follows that $R_{3}=A^{8s}$. Therefore
$R_{3}\in\langle R_{1},A\rangle$. Then with Equ. (\ref{eq:R3^2tau})
we get
\[
R_{2}=R_{1}^{7}R_{1}R_{2}R_{3}R_{3}^{2\tau-1}=R_{1}^{7}A(A^{8s})^{2k-1}\in\langle R_{1},A\rangle.
\]
This completes the second case. In the first case, i.e. when $\tau=1\mod4$
we find two solutions for $X^{2}=R_{3}^{8}$, namely $X=R_{3}^{4}$
and $X=R_{3}^{\tau+4}$. It is easy to check that 
\[
A^{4}=R_{3}^{\tau+4}.
\]
Choose $s\in\N$ such that $s(\tau+4)=1\mod2\tau$. In fact we can
give an explicit form for $s\in\N$. It depends on the residue $\tau\equiv1\mod8$
or $\tau\equiv5\mod8$. If $\tau=1+8\mu$, where $\mu\in\N$ we have
$s=6\mu+1$ (since $s\left(\tau+4\right)=(6\mu+1)(4+\tau)=(6\mu+1)(5+8\mu)=(3\mu+2)2\tau+1$)
and if $\tau=5+8\mu$, $\mu\in\N_{0}$ then $s=\left(14\mu+9\right)$
by a similar computation as above.

Then $A^{4s}=R_{3}$. The rest of the proof goes along the same lines
as in the previous case. 
\end{svmultproof}
For later use we note some simple relations
\begin{equation}\label{equ:gen_fix_space_1}
R_{2}^{2}=R_{1}A^{2}R_{1}^{3}A^{2}
\end{equation}
 and 
\begin{equation}\label{equ:gen_fix_space_2}
-R_{2}^{2}=R_{1}A^{2}R_{1}^{3}A^{2+k}.
\end{equation}

\begin{lemma}\label{lem:no_commut} 
There is no $\sigma\in\N$ such that
\[
AR_{1}^{2}=R_{1}^{2}A^{\sigma}.
\]
\end{lemma}
\begin{svmultproof}
Assume such a $\sigma\in\N$ exists, then 
\[
R_{1}^{2}A^{2+k}=A_{1}^{2}R_{1}^{2}=AR_{1}^{2}A^{\sigma}=R_{1}^{2}A^{2\sigma},
\]
which implies $\sigma=2\tau+1$, where $4\tau=k$. Since $A$ is explicitly
known, we can prove that the equation in our lemma does not hold.
For this we note that $R_1^6=-R_1^2$, so the third relation in Lemma \ref{lem:rels} reads $R_1^2 A^2=-A^2R_1^2$. Then the equation
\[AR_1^2=R_1^2A^{2\tau+1}\]
can be transformed to
\[AR_1^2=A^{2\tau}R_1^2A.\]
From the form $R_1^2$ in (\ref{equ:r1square}) and $A$ in (\ref{eq:represent4_A}) we see that $AR_1^2$ a structure with nontrivial $2\times 2$-blocks as in (\ref{eq:represent4_A}), while the right hand side has a $2\times 2$-block structure where the nontrivial blocks are located as indicated by the stars in the following matrix
\[\left(
  \begin{array}{cccc}
    0_2&0_2&0_2&*\\
    0_2&0_2&*&0_2\\
    0_2&*&0_2&0_2\\
    *&0_2&0_2&0_2 
  \end{array}
\right).\]
This proves that equality cannot hold. 
\end{svmultproof}
\section{The Abstract Generation Theorem}
In this section we provide the main technical result of this paper. It describes the structure of a set of groups which is derived from the relations satisfied by the generators. These relations are precisely the ones which we have derived for the generators of the groups $G(\ell),\ \ell\in\N$. 
\begin{theorem}{Theorem C (Abstract Generation Theorem)}\label{thm:Abstract-Generation-Result}
Let $R$ be a group with neutral element $e\in R$, let $r,a\in R$
be elements such that there exists a $k\in\N$, $k\ge12$ and $k\equiv4\mod8$
such that
\begin{enumerate}
\item $a^{2k}=e$, $a^{k}\in C(R)$ and $a^{k}=r^{4}$.
\item $(ar)^{2}=(ra)^{2}=(a^{3}r^{3})^{2}=e.$
\item $r^{2}a^{2}r^{2}=a^{2}$. 
\end{enumerate}
Then, we have
\begin{enumerate}
\item $r^{8}=\1$
\item $ar=r^{3}a^{k-1}$
\item $a^{3}r^{3}=ra^{k-3}$
\item $r^{2}a^{2}=a^{k+2}r^{2}$ 
\item $r^{2}a^{4}=a^{4}r^{2}$
\item $a^{4}r=ra^{2k-4}$
\item $a^{2}\in N_{R}\left(\langle r^{2}\rangle\right)$ and $r^{2}\in N_{R}(\langle a^{2}\rangle)$.
\end{enumerate}
Moreover let $G=\langle r,a\rangle$. Then, if $r^{2}a=a^{s}r^{2}$
for some $s\in\N_{0}$ then the order of $G$ is $8k,$ otherwise
it is $16k$. In order to describe the group more precisely we let
$\mathcal{A}=\langle a\rangle$ be the cyclic group generated by $a$.
Then in the first case $G$ consists of four cosets of $\mathcal{A}$,
i.e. 
\[
G=\mathcal{A}\cup r\mathcal{A}\cup r^{2}\mathcal{A}\cup r^{3}\mathcal{A}.
\]
In the second case $G$ is of order $16k$ and $G$ is the union of
$\mathcal{A}$ and the seven nontrivial cosets generated by $r,\ r^{2},\ r^{3},\ ar^{2},\ ar^{3},\ a^{2}r^{3},\ ra^{2}r^{3}$
respectively. 
\end{theorem}
We will give the proof in a separate section near the end of the paper. 
\section{Proof of Theorem B}
Before we start with the actual proof of Theorem B we state one more lemma which relates to the coset decomposition provided by Theorem C. Observe that we have seen in Lemma \ref{lem:no_commut} that our generators do not satisfy a relation of the form $R^2A=A^sR^2$ and therefore we are in the second case, where we have the eight cosets as described in Theorem C.  
\begin{lemma}
\label{lem:normalizer_matrices} Let $H=\langle R_{2}^{2}\rangle$
and $H'=\langle-R_{2}^{2}\rangle$, then $\dim\Fix{H}=\dim\Fix{H'}=4$,
and the normalizers of $H$ and $H'$ are equal, i.e. $N_{G(\ell)}(H)=N_{G(\ell)}(H')$. The normalizer $N_{G(\ell)}(H)$
is an index $2$ subgroup of $G(\ell)$ and the Weyl groups
\[
W(H)=N_{G(\ell)}(H)/H\ \mbox{ and }\ W(H')=N_{G(\ell)}(H')/H'
\]
are isomorphic. The groups $W(H),\ W(H')$ are isomorphic to $G_{3}(\tau)$.
They act absolutely irreducibly on the corresponding fixed point subspaces.
Moreover we have
\[
\Fix{H}\oplus\Fix{H'}=\R^{8}.
\]
\end{lemma}
\begin{svmultproof}
Using Equations (\ref{equ:gen_fix_space_1},\ref{equ:gen_fix_space_2}) we represent $\pm R_2^2$ as $R_1A^2R_1^3A^{2}$ and $R_1A^2R_1^3A^{2+\tau}$ respectively.  
From this it is easy
to see that the group generated by $R_{1}^{2},\ A^{2}$ and $R_{1}A$
commutes with $R_{2}^{2}$ and with $-R_{2}^{2}$. We could derive the properties of the elements $\pm R_2^2$ and the commutation
relation from the abstract setting of Theorem C (Abstract Generation Theorem), however the resulting computation is
slightly more involved, than making use of the matrices, compare Lemma \ref{lem:normalizer_abstract}. Note the form of $R_{2}^{2}$:
\[
R_{2}^{2}=\left(\begin{array}{cc}
\1_4 & 0_4 \\
0_4 & -\1_4\\
\end{array}\right).
\]
So on the matrix level it is obvious that any $4\times4$-block-diagonal
matrix commutes with $R_{2}^{2}$ and with $-R_{2}^{2}.$ From the
form of $R_{2}^{2}$ we get also immediately the statements on the
fixed point subspaces. In order to prove the statements on the action
of the Weyl group we have to look a bit closer. According to Equation
$\left(\ref{equ:G3_generation}\right)$ the group $G_{3}(m)$ is generated
by four elements, in the biquaternionic notation these are $[e_{m},1],\ [1,i],\ [1,j]$
and $[j,1]$ for $m=3,\ ,5,\dots$. In \cite{LM} it is shown, that
we can replace the first two elements by their product and we can
exchange the term $[j,1]$ by $[j,j]$. So we have the generators
\begin{equation}
\theta_{1}=[e_{m},i],\ \theta_{2}=[1,j],\ \theta_{3}=[j,j],\ m\in2\N+1.\label{eq:G3generation_short}
\end{equation}
For $j=1,2,3$ we define $4\times4$-matrices $M_{j}$ by
\[
M_{j}\left(\begin{array}{c}
v_{1}\\
v_{2}\\
v_{3}\\
v_{4}
\end{array}\right)=\theta_{j}\left(v_{1}+iv_{2}+jv_{3}+kv_{4}\right).
\]
From the fact that all $4\times 4$-block diagonal matrices commute with $\pm R_2^2$ 
and $R^1R_2^2R_1\inv =-R_2^2$ we conclude that the normalizer
 $N_{G(\ell)}\left(H\right)$ is given  
precisely by the set of  block-diagonal
matrices  with $4\times4$-blocks on the diagonal. 
Now we determine $\Omega_{j}\in N_{G(\ell)}\left(H\right)$ with 
\begin{equation}
\Omega_{j}=\left(\begin{array}{cc}
U_{j,1} & 0_{4}\\
0_{4} & U_{j,2}
\end{array}\right),\mbox{ with }U_{j,1}=M_{j}\mbox{ for }j=1,2,3.\label{eq:represent_determination}
\end{equation}
Observe, that $\Omega_{j}$ acts on $\Fix{H}$ as $\theta_{j}$, so
if we solve Equation $\left(\ref{eq:represent_determination}\right)$
it follows that the quotient $N_{G(\ell)}(H)/H$ is isomorphic to
$G_{3}(\tau)$, where $k=4\tau$ and the action is given by the one
which we have described in \cite{LM}. We write the generators of
$N_{G(\ell)}(H)$ as $\Xi_{j},\ j=1,2,3$ with
\[
\Xi_{1}=AR_{1},\ \Xi_{2}=R_{1}^{2},\ \Xi_{3}=A^{2}.
\]
 From the decomposition of $G(\ell)$ into cosets, as it is abstractly
described in Theorem C it follows
(using the relations) that any element $g\in N_{G}(H)$ has the (unique)
form
\[
g=\Xi_{1}^{\nu_{1}(g)}\Xi_{2}^{\nu_{2}(g)}\Xi_{3}^{q(g)},\mbox{ with }\nu_{1}(g),\ \nu_{2}(g)\in\lbrace0,1\rbrace,\ q(g)\in\lbrace2,4,\dots,\ell\rbrace.
\]
By a lengthy, but explicit computation, we can prove, that Equ.
$\left(\ref{eq:represent_determination}\right)$ has the solutions
(which depend on $\tau$)
\[
\Omega_{1}=\Xi_{2}\Xi_{3}^{q_{1}\left(\tau\right)},\ \Omega_{2}=\Xi_{3}^{q_{2}(\tau)},\ \Omega_{3}=\Xi_{1}\Xi_{2}\Xi_{3}^{q_{3}(\tau)}
\]
with
\[
q_{1}\left(\tau\right)=\begin{cases}
2\tau+2 & ,\mbox{ if }\tau\equiv3\mod4\\
8\tau-16 & ,\mbox{ if }\tau\equiv1\mod4,
\end{cases}\]
\[
q_{2}\left(\tau\right)=\begin{cases}
3\tau & ,\mbox{ if }\tau\equiv3\mod4\\
\tau & ,\mbox{ if }\tau\equiv1\mod4
\end{cases}
\]
 and
\[
q_{3}(\tau)=\frac{c}{2}\tau+\frac{1}{2},
\]
where
\[
c=\begin{cases}
5 & ,\mbox{ if }\tau\equiv1\mod8\\
3 & ,\mbox{ if }\tau\equiv3\mod8\\
1 & ,\mbox{ if }\tau\equiv5\mod8\\
7 & ,\mbox{ if }\tau\equiv7\mod8
\end{cases}
\]
\end{svmultproof}
\begin{proof}[of Theorem B]
Let $H,\ H'$ be as in the previous lemma.
\begin{enumerate}
\item Follows from the abstract generation theorem (Theorem C). We see from the lemmas \ref{lem:rel1} to \ref{lem:gens}  that the generators $R,\ A$ satisfy the relations required for Theorem C. Moreover Lemma \ref{lem:no_commut} implies that we are in the case with eight cosets of the cyclic group generated by $A$ and therefore we  $G(\ell)$ has $16k=64+128\ell$ elements.  
\item If $T:\R^{8}\to\R^{8}$ is a linear map which commutes with $G(\ell),$
then $T:\Fix{H}\to\Fix{H}$ and $T:\Fix{H'}\to\Fix{H'}$ and commutes
with the normalizers and hence also with the respective Weyl groups.
Since the Weyl groups act absolutely irreducibly on these fixed point
spaces, the restriction of $T$ to any of the fixed point spaces is
a multiple of the identity. So in general $T=c_{1}\1_{\Fix{H}}+c_{2}\1_{\Fix{H'}}$.
Now, $T$ has to commute with $R_{1}$, so we have
\[
R_{1}T=\left(\begin{array}{cc}
0_{4} & \F\\
\F & 0_{4}
\end{array}\right)\left(\begin{array}{cc}
c_{1}\1_{4} & 0_{4}\\
0_{4} & c_{2}\1_{4}
\end{array}\right)=\left(\begin{array}{cc}
0_{4} & c_{1}\F\\
c_{2}\F & 0_{4}
\end{array}\right)=\left(\begin{array}{cc}
0_{4} & c_{2}\F\\
c_{1}\F & 0_{4}
\end{array}\right)=TR_{1}.
\]
Therefore we have 
\[
c_{1}=c_{2}
\]
and
\[
T=c\1_{8}.
\]
This implies absolute irreducibility.
\item If $k_{1}|k_{2}$ then the groups generated by $A(k_{1})$ and $A(k_{2})$
are cyclic and clearly the first one is a subgroup of the second one.
Since we have $G(\ell)=\langle R_{1},A\rangle$ the inclusion for
the groups comes from the inclusion of the cyclic groups $\langle A(k_{j})\rangle,\: j=1,2$.
Then if $\ell_1,\ell_2\in\N$, are given we look at $k_j=4+8\ell_j=4\tau_j$, where $\tau_j$ $j=1,2$ are odd. 
Clearly we have for $k=4\tau_1\tau_2$ that $k=4+8\ell$, where $2\ell+1=\tau_1\tau_2$ 
\[G(\ell_j)\subset G(\ell), \mbox{ for } j=1,2.\] 
This implies that the union of $G(\ell),$ $\ell\in\N$ is a group and therefore its closure is a compact Lie group.
\item If $K<G\left(\ell\right)$ is an isotropy subgroup then we distinguish
several cases. 
\begin{enumerate}
\item $H\subset K$ and $H'\subset K$
\item One of these groups is contained in $K$, the other one is not included in $K$.
\item None of the group $H,\ H'$ is contained in $K$.
\end{enumerate}
We look at these cases one by one.
\begin{enumerate}
\item If both groups $H$ and $H$' are contained in $K$, then 
\[
\Fix{K}\subset\Fix{H}\cap\Fix{H'}=\lbrace0\rbrace.
\]
\item If just one of the groups $H,\ H'$ is contained in $K$, w.l.o.g.
say $H\subset K$, then 
\[
\Fix{K}\subset\Fix{H}.
\]
Let $\pi:N_{G(\ell)}(H)\to W_{G(\ell)}(H)$ be the canonical projection and let
$K'=K\cap N_{G(\ell)}(H)$. Then 
\[
\Fix{K}\subset\Fix{\pi(K')}\subset\Fix{H}
\]
and the spaces on the right hand side are even dimensional, since
$G_{3}(\tau)$ in its irreducible four dimensional representation
has only even dimensional fixed point spaces. If $K'=K$ then $\Fix{K}=\Fix{\pi(K)}$
is even dimensional. Assume $K\cap(G(\ell)\setminus N_{G(\ell)}(H))\not=\emptyset$.
Let $V\in K\cap(G\setminus N_{G}(H))$. By inspection we see that
$V$ has block antidiagonal form as
\[
V=\left(\begin{array}{cc}
0_{4} & V_{1}\\
V_{2} & 0_{4}
\end{array}\right),
\]
 where $V_{1,2}$ are linear maps $\Fix{H}\to\Fix{H'}$ and vice versa.
Therefore $V$ cannot fix an element in $\Fix{H}$. It follows that
$\dim(\Fix{K})$ is even. 
\item In this case the fixed point space $\Fix{K}$ has nontrivial projections
into $\Fix{H}$ and into $\Fix{H'}$ and $\Fix{H}\cap\Fix{K}=\lbrace0\rbrace=\Fix{H'}\cap\Fix{K}$.
Let $x\in\Fix{K}$. Since for $L\in N_{G}\left(H\right)$ we have
$L:\Fix{H}\to\Fix{H}$ and $L:\Fix{H'}\to\Fix{H'}$ it follows that
$L$ fixes the projections $\rho(x),\ \rho'(x)$, where 
\[
\rho:\R^{8}\to\Fix{H},\ \ker\rho=\Fix{H'}
\]
and
\[
\rho':\R^{8}\to\Fix{H'},\ \ker\rho'=\Fix{H}
\]
are projection operators. By the previous consideration $L$ has an
even dimensional fixed point space of the form $F\oplus F'$ where
$F\subset\Fix{H}$ and $F'\subset\Fix{H'}$. Therefore elements in
$N_{G(\ell)}\left(H\right)$ contribute even dimensional fixed point spaces.
Since
\[
\Fix{K}=\bigcap_{L\in K}\Fix{L}
\]
we can only get odd dimensional fixed point spaces if $K$ contains
elements outside this normalizer. Let $V\in K$, $V\notin N_{G}\left(H\right)$.
Then $V$ has the form as described before, i.e.
\[
V=\left(\begin{array}{cc}
0_{4} & V_{1}\\
V_{2} & 0_{4}
\end{array}\right).
\]
 Then $x\in\R^{8}$ is fixed under $V$, if and only if $V_{1}(\rho'(x))=\rho(x)$
and $V_{2}(\rho(x))=\rho'(x).$ This implies
\[
V^{2}(\rho(x))=\rho(x)
\]
 and 
\[
V^{2}(\rho'(x))=\rho'(x).
\]
 But then $V^{2}\in N_{G}\left(H\right)$ and $Q=\Fix{V^{2}}\cap\Fix{H}$
is even dimensional. Then
\[
\Fix{V}=\Meng{(q,V(q)}{q\in Q}
\]
is even dimensional. Therefore $\dim(\Fix{K})$ is even.
\end{enumerate}
\end{enumerate}
\end{proof}
\begin{remark}{Remark}
  It is instructive and useful to compare these proofs with computations using GAP \cite{GAP}. For small values of $\ell$ the GAP-names of these groups are 
\[G(1)=[192,36],\;G(2)=[320,35],\;G(3)=[448,34],\mbox{ and }G(4)=[576,37].\]
\end{remark}
\section{Proof of Theorem C}\label{sec:pf_thm_C}
\begin{proof}
Let $\mathcal{A}=\langle a\rangle$ denote the cyclic subgroup generated
by $a$, its order is $2k$. If we show that $G$ is given by the
seven cosets, then its order is obviously $16k$. Let us denote the
cosets by $C_{j},\ j=0,\dots,7$ with
\[
C_{0}=\mathcal{A},\ C_{1}=r\mathcal{A},\ C_{2}=r^{2}\mathcal{A},\ C_{3}=r^{3}\mathcal{A},\ C_{4}=ar^{2}\mathcal{A},\ C_{5}=ar^{3}\mathcal{A},\ C_{6}=a^{2}r^{3}\mathcal{A},\ \mbox{{and}}
\]
\[
C_{7}=ra^{2}r^{3}\mathcal{A}.
\]
 Let us go through the list of assertions. 
\begin{enumerate}
\item We have 
\[
r^{8}=\left(r^{4}\right)^{2}=\left(a^{k}\right)^{2}=a^{2k}=e.
\]

\item If $\left(ra\right)^{2}=e,$ then using the first hypothesis we have
\begin{equation}
r^{3}a^{k-1}=r^{3}ea^{k-1}=r^{3}raraa^{k-1}=r^{4}(ar)a^{k}=ar,\label{eq:rel1}
\end{equation}

\item and similarly 
\begin{equation}
ra^{k-3}=rea^{k-3}=rr^{3}a^{3}r^{3}a^{3}a^{k-3}=r^{4}a^{3}r^{3}a^{k}=a^{3}r^{3}.\label{eq:rel3}
\end{equation}

\item Again in a very similar way we find
\begin{equation}
r^{2}a^{2}=r^{2}r^{2}a^{2}r^{2}=a^{k+2}r^{2}\label{eq:rel2}
\end{equation}

\item and
\[
r^{2}a^{4}=r^{2}a^{2}a^{2}=a^{k+2}r^{2}a^{2}=a^{k+2}a^{k+2}r^{2}=a^{4}r^{2}.
\]

\item For the additional relation we see
\begin{equation}
a^{4}r=a^{3}(ar)=a^{3}r^{3}a^{k-1}=ra^{2k-4}.\label{eq:induction}
\end{equation}

\item The last statement in our list follows immediately from the fourth
statement.
\end{enumerate}
Since $a^{2k-4}=\left(a^{4}\right)^{-1}$ but $a^{4}\not=a^{2k-4}$
(since $k>4$) $G$ is not abelian. Especially $ar\not=ra$. 

Next we observe that $C_{0},$ $C_{1},\ C_{2},\ C_{3}$ are disjoint
cosets of $\mathcal{A}.$ Assume two of these cosets were equal. Then
we have an equation
\[
r^{q}=r^{p}a^{s}\mbox{ for }p,\ q\in{0,1,2,3},\ s\in{1,\dots,2k-1},p\not=q.
\]
We have either $p>q$ or $p<q.$ We look at the first case (the second
is similar), then $e=r^{p-q}a^{s}.$ $0<p-q\le3$ implies $r^{p-q}\notin\mathcal{A}$
but $r^{p-q}=(a^{s})^{-1}\in\mathcal{A}$.

Write $G=\langle a,r\rangle$ for the group generated by the elements
$r$, $a$ in $R$. Obviously we have $8\,\big|\,|G|$ and $2k\,\big|\,|G|$.
Since the cosets $C_{0}$, $C_{1}$, $C_{2}$, $C_{3}$ are disjoint
the group $G$ has at least $8k$ elements. 

Observe first that $ar\mathcal{A}=r^{3}a^{k-1}\mathcal{A}=r^{3}A=C_{3}$
does not give an additional coset. 

If $ar^{2}=r^{2}a^{s}$ for some $s\ge0$ then the coset defined by
$ar^{2}$ is the same one as the one defined by $r^{2}.$ In a similar
way we have $a^{2}r^{3}=ar^{3}a^{k-1}r^{2}=ar^{3}ar^{2}a^{(k+2)(k-2)/2}=ar^{3}r^{2}a^{s}a^{k^{2}/2-1}=\in ar\mathcal{A}=r^{3}\mathcal{A}$.
Therefore the coset of $a^{2}r^{3}$ is the same as the one of $r^{3}$
and $ra^{2}r^{3}\in\mathcal{A}.$ The last coset to be looked at is
the one defined by $ar^{3}=r^{3}a^{k-1}r^{2}=r^{5}a^{q}$ for some
$q\in\N$, which is the coset of $r$.

In the second case the coset of $ar^{2}$ is different from those
defined before and we have a new coset. The same applies to the cosets
of $ar^{3}$, $a^{2}r^{3}$ and $ra^{2}r^{3}$. 

From now on we treat the second case only since, by Lemma \ref{lem:no_commut},
this is the one which appears in $G(\ell)$. 

Let $w=w(r,a)$ denote any word in $r,a$, we first want to show that
\begin{equation}
w\mathcal{A}\subset\bigcup_{j=0}^{7}C_{j}.\label{eq:union}
\end{equation}
 W.l.o.g. $w(r,a)$ ends with a power of $r$, otherwise we can rewrite
it by putting the $a$-power into $\mathcal{A}$. Since $r^{4}\in\mathcal{A}$
the word can be shortened such that the end is a power $r^{s},\ s\in\lbrace1,2,3\rbrace$
with exponent at most three. Moreover the word consists of a product
of any powers of $a$ and powers of $r^{s}\ s\in\lbrace1,2,3\rbrace$.
Let $m=m(w)$ denote the number of powers of $r$ appearing in the
product describing the word $w$. We prove by induction on $m\in\N$
that Equation (\ref{eq:union}) is satisfied.

We begin with $m=1$, i.e. $w(r,a)=a^{j}r^{s}$, where $j\in\N$ and
$s\in{1,2,3}$. Again we prove this by (a somewhat unusual) induction
on $j\in\N$ and for $s=1$ we have already seen the first step. We
consider the cases for the different $s\in\lbrace1,2,3\rbrace$ separately.
Let us first do some simple cases which will be needed for the induction,
where we use the relations $\left(\ref{eq:rel1},\ref{eq:rel3}\right)$:

\[
a^{2}r=ar^{3}a^{k-1}\in C_{5}
\]

\[
a^{3}r=a^{2}r^{3}a^{k-1}\in C_{6}
\]
and again with $\left(\ref{eq:rel3}\right)$
\[
a^{4}r=a^{3}r^{3}a^{k-1}=ra^{2k-4}\in C_{1}.
\]

Similar computations show that $a^{j}r^{3}\in C_{1}\cup\ C_{3}\cup\ C_{5}\cup\ C_{6}$
for $j=1,2,3$. Given $j\in\N,$$j\ge4$. Let $j_{0}\in\N$ be maximal
with $4j_{0}\le j$. Then $j=4j_{0}+q$, where $\N\ni0\le q\le3$.
Then 
\[
a^{j}r=a^{4j_{0}+q}r=a^{q}(a^{4})^{j_{0}}r=a^{q}r\left(a^{2k-4}\right)^{j_{0}}\in a^{q}r\mathcal{A}\subset C_{1}\cup C_{3}\cup C_{5}\cup C_{6}.
\]
A similar computation applies to $s=3.$ If $s=2$ we have for odd
$2j+1\in\N$
\[
aa^{2j}r^{2}=ar^{2}a^{2j}\in C_{4}.
\]
The case $s=2$, $2j\in\N$ is still simpler. This finishes the case
$m=1$. 

In the general case $m\in\N$ we look at the word $w=a^{s_{1}}r^{j_{1}}\dots a^{s_{m}}r^{j_{m}}$.
W.l.o.g. we assume that $j_{r}\le3,\ r=1,\dots,m$, since $r^{4}\in\mathcal{A}.$
We rewrite $w$ in the form 
\[
w=w_{1}r^{j_{m-1}}a^{s_{m}}r^{j_{m}},
\]
where $w_{1}$ is the first part of the word. Again it is no loss
of generality to assume that $0\le s_{m}\le3$, since the induction
relation (\ref{eq:induction}) allows the following reduction step
for $s_{m}\ge4$
\[
r^{j_{m-1}}a^{s_{m}}r^{j_{m}}=r^{j_{m-1}}a^{s_{m}-4}r^{j_{m}}a^{j_{m}(2k-4)}\in r^{j_{m-1}}a^{s_{m}-4}r^{j_{m}}\mathcal{A}.
\]
 This leads us to study the following $27$ cases: $1\le j_{m-1},\ j_{m}\le3$,
$1\le s_{m}\le3$. The Table \ref{tab:cosets} gives the results:

\begin{table}[H]
\begin{onehalfspace}
\begin{centering}
$\begin{array}{|c|c|c|c||c|c|c|c|}
\hline j_{m-1} & j_{m} & s_{m} & \mbox{coset} & j_{m-1} & j_{m} & s_{m} & \mbox{coset}\\
\hline\hline 1 & 1 & 1 & \mathcal{A} & 2 & 2 & 3 & ra^{2}r^{3}\mathcal{A}\\
\hline 1 & 1 & 2 & ar^{2}\mathcal{A} & 2 & 3 & 1 & a^{2}r^{3}\mathcal{A}\\
\hline 1 & 1 & 3 & ra^{2}r^{3}\mathcal{A} & 2 & 3 & 2 & ar^{2}\mathcal{A}\\
\hline 1 & 2 & 1 & a^{2}r^{3}\mathcal{A} & 2 & 3 & 3 & r^{3}\mathcal{A}\\
\hline 1 & 2 & 2 & r^{3}\mathcal{A} & 3 & 1 & 1 & r^{2}\mathcal{A}\\
\hline 1 & 2 & 3 & a^{2}r^{3}\mathcal{A} & 3 & 1 & 2 & ra^{2}r^{3}\mathcal{A}\\
\hline 1 & 3 & 1 & ar^{2}\mathcal{A} & 3 & 1 & 3 & ar^{2}\mathcal{A}\\
\hline 1 & 3 & 2 & ra^{2}r^{3}\mathcal{A} & 3 & 2 & 1 & ar^{3}\mathcal{A}\\
\hline 1 & 3 & 3 & r^{2}\mathcal{A} & 3 & 2 & 2 & r\mathcal{A}\\
\hline 2 & 1 & 1 & r\mathcal{A} & 3 & 2 & 3 & ar^{3}\mathcal{A}\\
\hline 2 & 1 & 2 & a^{2}r^{3}\mathcal{A} & 3 & 3 & 1 & ra^{2}r^{3}\mathcal{A}\\
\hline 2 & 1 & 3 & ar^{2}\mathcal{A} & 3 & 3 & 2 & ar^{2}\mathcal{A}\\
\hline 2 & 2 & 1 & ra^{2}r^{3}\mathcal{A} & 3 & 3 & 3 & \mathcal{A}\\
\hline 2 & 2 & 2 & \mathcal{A} & - & - & - & -
\\\hline \end{array}$
\par\end{centering}
\end{onehalfspace}

\centering{}\caption{\label{tab:cosets}Cosets and the last part of the word.}
\end{table}

From the table we see that in most cases we reduce the number $m$
by at least one. The exception are the six cases cases where we get
the coset $ra^{2}r^{3}\mathcal{A}$. After doing this reduction we
look at the part of the word which has the form
\[
r^{j_{m-2}}a^{s_{m-1}}ra^{2}r^{3}a^{\sigma}
\]
where $j_{m-2},\ s_{m-1}\in\lbrace1,2,3\rbrace$ and $\sigma\in\N$.
Again we display the results in a table, see Table \ref{tab:cosets_complic}.

\begin{table}[H]
\begin{centering}
\begin{tabular}{|c|c|c|}
\hline 
$j_{m-2}$ & $s_{m-1}$ & coset\tabularnewline
\hline 
\hline 
$1$ & $1$ & $r^{3}\mathcal{A}$\tabularnewline
\hline 
$1$ & $2$ & $a^{2}r^{3}\mathcal{A}$\tabularnewline
\hline 
$1$ & $3$ & $r^{3}\mathcal{A}$\tabularnewline
\hline 
$2$ & $1$ & $\mathcal{A}$\tabularnewline
\hline 
$2$ & $2$ & $ra^{2}r^{3}\mathcal{A}$\tabularnewline
\hline 
$2$ & $3$ & $\mathcal{A}$\tabularnewline
\hline 
$3$ & $1$ & $r\mathcal{A}$\tabularnewline
\hline 
$3$ & $2$ & $ar^{3}\mathcal{A}$\tabularnewline
\hline 
$3$ & $3$ & $r\mathcal{A}$\tabularnewline
\hline 
\end{tabular}
\par\end{centering}

\caption{\label{tab:cosets_complic}In this table we look a the words where
the first step did not lead to a simplification.}
\end{table}

So, in each case we have reduced the number factors which are a power
of $r$, by at least one, and the remaining word has the same structure,
so by induction each word can be transformed using the relations into
a member of one the eight cosets. Therefore we have established the
theorem. \end{proof}
\begin{lemma}\label{lem:normalizer_abstract}
We look at $h=ra^{2}r^{3}a^{2}$ and $h'=ra^{2}r^{3}a^{2+k}$. Then
$\langle h\rangle=H\subset G$ as well as $\langle h'\rangle=H'\subset G$
is a subgroup, whose respective normalizer $N_{G}(H)=$$N_{G}(H')\subset G$
has index $2$ in $G$. 
\end{lemma}
\begin{proof}
Let $\mathcal{A}_{e}\subset\mathcal{A}$ be the set of even powers
of $a$ and $\mathcal{A}_{o}$ be the set of odd powers of $a.$ Consider
the subgroup $K=\langle r^{2},\ ar,\ a^{2}\rangle$ of $G.$ It is
a bit tedious to show that each of the generators commutes with $h$
or $h'$ respectively (in fact the computations are much more involved
than the corresponding computations in the case of the matrices).
We claim that $K$ is explicitly given by 
\[
K=\mathcal{A}_{e}\cup r\mathcal{A}_{o}\cup r^{2}\mathcal{A}_{e}\cup r^{3}\mathcal{A}_{o}\cup ar^{2}\mathcal{A}_{o}\cup ar^{3}\mathcal{A}_{e}\cup a^{2}r^{3}\mathcal{A}_{o}\cup ra^{2}r^{3}\mathcal{A}_{e}.
\]
This proves the claim on the order of $K$.
\end{proof}
\section{Bifurcations for the new series\label{sec:Bifurcations8}}
First we look at the equivariant structure. The main point here is,
that we cannot explicitly compute equivariants and we can only compute
dimensions for homogeneous equivariant polynomial maps using character
theory and this is limited to small degrees and to the first few groups.
In order to obtain general statements we start with the groups $G_{3}(\tau)$.
In our context we have an action of this group as the Weyl group on the
fixed point spaces of $H$ and $H'$ respectively.  From the results
in Section \ref{sec:Bifurcations4} we can determine cubic $W(N_{G(\ell)}(H)/H)$
equivariant maps on these fixed point spaces. The elements in $G(\ell)\setminus N_{G(\ell)}(H)$
map those euivariants onto each other. In this way we can construct
cubic equivariant maps which restrict onto on the fixed point spaces.
Given a group $G$ acting on a real vector space $V$ we denote by
$C_{G}^{\infty}(V,V)$ the space of smooth $G$-equivariant maps $V\to V$,
by $\mathcal{E}_{G}(V,V)$ we denote the space of equivariant polynomial
mappings and by $\mathcal{E}_{G}^{\mu}(V,V)$ we denote the subspace
of $\mathcal{E}_{G}(V,V)$ of such maps which are homogeneous of degree
$\mu\in\N$. If $W\subset V$ is a fixed point subspace of a subgroup
$K\subset G$ we look at the restriction map 
\[
\Pi_{\mu}:\mathcal{E}_{G}^{\mu}(V,V)\to\mathcal{E}_{N_{G}(K)}(W,W):p\mapsto p_{\big|_{W}}.
\]
In general this map is neither injective nor surjective. However we
have the following lemma.
\begin{lemma}
\label{lem:surjective_restriction}For $K\in\lbrace H,H'\rbrace$
the mapping
\[
\Pi_{3}:\mathcal{E}_{G(\ell)}^{3}(\R^{8},\R^{8})\to\mathcal{E}_{N_{G(\ell)}(K)}^3(\Fix{K},\Fix{K})
\]
is surjective.\end{lemma}
\begin{proof}
Follows directly from the above construction.
\end{proof}
With this lemma we can prove the main result.
\begin{proposition}
The representation of $G(\ell)$ on $\R^{8}$ is generically symmetry
breaking in the sense of {\rm \cite{Fd6}}. The set of nontrivial branches
includes steady states in one of the two dimensional fixed point spaces
in $\Fix{H}$.
\end{proposition}
\begin{proof}
The generic cubic map on $\R^8$ restricts to the fixed point space $\Fix{H}$, respectively to $\Fix{H'}$ to be the generic $W(H)$ (or $W(H')$)-equivariant cubic map. Therefore the bifurcation behavior is the same as the one described in Theorem A. 
\end{proof}
\section{Remarks on further  series in $\R^{4}$ and $\R^8$}\label{sec:Remarks}
In the introduction we have made some remarks on extending the three
families defined in \cite{LM} to include more groups. The series
defined there include groups of order $16m$, where $m$ is odd. If
we look at the GAP computations given \cite{LM}, Table 5 we see there are
many more groups which provide counter examples to the (AIC). Looking
at the group orders we have additional groups 
\begin{enumerate}
\item of orders $2^{k}m$, where $m$ is odd and $k>4$ and 
\item of the form $8m$, where $m$ is odd, square-free and non-prime.
\end{enumerate}
In the second case Table 5 from \cite{LM} gives precisely three groups.
For these it is relatively easy to show that there is a natural extension  of definition of the three families ${\mathcal{F}}_j$, $j=1,2,3$. We  can use the definition
of the generators as given \cite{LM} for the original families. In
the first case and in the case $k=4$ the situation is slightly more
complicated. If $m$ is square-free but non-prime in most cases there
are more than three groups which fail (AIC). In the case where the prime
factorization of $m$ contains squares the problem is split: for $144=2^{4}\cdot9$ $m=9$ 
is not square-free but there are more than three such groups of order $144$, for $400=2^{4}\cdot25$
there only three such groups which are in the families we have described.
For the four dimensional situation we have a few cases where we do
not have a complete picture of the relation of the failure of (AIC)
and Ize's conjecture, but in most cases where we have the failure
of (AIC) Ize's conjecture remains true. So it seems that the Ize conjecture
is true in the case of dimension $4$. \\[3mm]
In $\R^8$ we have numerical evidence for the existence of two series of groups of order $16m$, where $m$ is odd, non-prime, square-free and contains at least one prime factor which is of form $1\mod 4$ (i.e. $m=15$, $35$, $39$, $\dots$) where there are only four dimensional fixed point spaces. Bifurcations and dynamics in these cases are subject for further research.   
\begin{acknowledgements}
I would like to thank Haibo Ruan for helpful discussions.  I thank Y. Krasnov for spotting  a computational error in a draft version of this paper. 
I would also like to thank the unknown referee for several hints and useful remarks. 
\end{acknowledgements}
\bibliographystyle{ab1}
\addcontentsline{toc}{section}{\refname}\bibliography{bif_a,bif_b,bif_c,bifurcat,bif_d,bif_f,bif_g,bif_h,bif_k,bif_m,bif_v}

\end{document}

%% file: macro.tex
 3
 2

\def\bigspace{\advance\lineskip by 3pt
      \advance\baselineskip by 3pt
      \advance\lineskiplimit by 3pt}
\global\def\sectitle#1\par{\bigbreak
  \leftline{\bf #1}
  \nobreak\medskip\vskip-\parskip
  \message{#1}
  \noindent}
\global\def\ssectitle#1\par{\bigbreak\medskip
  \leftline{\typc #1}
  \nobreak\bigskip\vskip-\parskip
  \message{#1}
  \noindent}
\global\def\sssectitle#1\par{\bigbreak\bigskip
  \leftline{\typd #1}
  \nobreak\bigskip\vskip-\parskip
  \message{#1}
  \noindent}
\global\def\sssectitletwo#1#2#3\par{\bigbreak\bigskip
  \vbox{
  \leftline{\typd #1#2}
  \nobreak\vskip2truemm
  \message{#1#2}
  \leftline{\typd \phantom{#1}#3}
  \message{#3}
  }
  \nobreak\bigskip\vskip-\parskip
  \noindent}
\def\mod{\mathop{\rm mod}}

\def\dim{\mathop{\rm dim}}
\def\Fix#1{\mathop{\rm Fix}(#1)}

\def\inv{^{-1}}

\def\Meng#1#2{\left\{#1\;\Big|\;#2\right\}}

\def\qed{\vbox{\hrule
  \hbox{\vrule\hbox to 5pt{\vbox to 8pt{\vfil}\hfil}\vrule}\hrule}}

\let\leftv=|
\let\rightv=|

\def\transversal{\kern0.5em\cap\kern-1em\stackrel{\textstyle\top}{}\kern0.25em}
\def\OO#1{\mathop{{\bf O}(#1)}}

\def\SO#1{\mathop{{\bf SO}(#1)}}

\def\1{{\rm 1\hskip-0.9truemm l}}

\def\ifundefined#1{\expandafter\ifx\csname
  #1\endcsname\relax}
\ifundefined{textheight}
\def\newline{\par\noindent}
\else
\fi
